\documentclass[a4paper,10pt,reqno]{amsart}

\usepackage{amsthm,amsmath,amssymb}
\usepackage{amsfonts}
\usepackage[T1]{fontenc}
\usepackage{enumitem}
\usepackage{caption}
\usepackage{graphicx}
\usepackage[breaklinks]{hyperref} 
\usepackage{comment}
\usepackage{pgf,tikz}
\usepackage{ifthen}
\usepackage{color}
\usetikzlibrary{arrows.meta,decorations.pathmorphing,decorations.pathreplacing,calc,intersections}
\usepackage{pgfplots}
\usepgfplotslibrary{fillbetween}
\tikzset{every picture/.style={line width=0.55pt}}%
\tikzset{>={Classical TikZ Rightarrow[scale=1.66]}}
\tikzset{
  show curve controls/.style={
	postaction={
	  decoration={
		show path construction,
		curveto code={
		  \draw [blue,opacity=0.85, line width=0.2pt] 
			(\tikzinputsegmentfirst) -- (\tikzinputsegmentsupporta)
			(\tikzinputsegmentlast) -- (\tikzinputsegmentsupportb);
		  \fill [red, opacity=0.45] 
			(\tikzinputsegmentsupporta) circle [radius=1pt]
			(\tikzinputsegmentsupportb) circle [radius=1pt];
		}
	  },
	  decorate
}}}
\definecolor{linkcolor}{rgb}{0,0,0.675}%
\hypersetup{colorlinks=true,linkcolor=linkcolor,citecolor=linkcolor,urlcolor=linkcolor,bookmarksdepth=2}%
\newtheorem{theorem}{Theorem}[section]
\newtheorem{lemma}[theorem]{Lemma}
\newtheorem{corollary}[theorem]{Corollary}
\newtheorem{proposition}[theorem]{Proposition}
\theoremstyle{definition}
\newtheorem{remark}[theorem]{Remark}
\newtheorem{definition}[theorem]{Definition}
\newtheorem{example}[theorem]{Example}

\setitemize{topsep=0pt plus 3pt,itemsep=0pt,leftmargin=25pt,listparindent=\parindent}
\setenumerate{topsep=0pt plus 3pt,itemsep=0pt,leftmargin=25pt,listparindent=\parindent}
\setitemize{topsep=0pt plus 3pt,itemsep=0pt,leftmargin=25pt,listparindent=\parindent}
\setenumerate[1]{label=\normalfont(\arabic{enumi})}
\setenumerate[2]{label=\normalfont(\alph{enumii}),leftmargin=25pt}
\setdescription{font=\sffamily, itemsep=0pt,leftmargin=8pt,listparindent=\parindent}

\makeatletter\@addtoreset{equation}{section}\makeatother
\newcommand*{\image}{\operatorname{im}}
\newcommand*{\Pic}{\operatorname{Pic}}
\renewcommand*{\div}{\operatorname{Div}}
\newcommand*{\xto}[1]{\xrightarrow{\,#1\,}}

\makeatletter
\newcommand*{\da@rightarrow}{\mathchar"0\hexnumber@\symAMSa 4B }
\newcommand*{\da@leftarrow}{\mathchar"0\hexnumber@\symAMSa 4C }
\newcommand*{\xdashrightarrow}[2][]{%
  \mathrel{%
	\mathpalette{\da@xarrow{#1}{#2}{}\da@rightarrow{\,}{}}{}%
  }%
}
\newcommand{\xdashleftarrow}[2][]{%
  \mathrel{%
	\mathpalette{\da@xarrow{#1}{#2}\da@leftarrow{}{}{\,}}{}%
  }%
}
\newcommand*{\da@xarrow}[7]{%
  \sbox0{$\ifx#7\scriptstyle\scriptscriptstyle\else\scriptstyle\fi#5#1#6\m@th$}%
  \sbox2{$\ifx#7\scriptstyle\scriptscriptstyle\else\scriptstyle\fi#5#2#6\m@th$}%
  \sbox4{$#7\dabar@\m@th$}%
  \dimen@=\wd0 %
  \ifdim\wd2 >\dimen@
	\dimen@=\wd2 %
  \fi
  \count@=2 %
  \def\da@bars{\dabar@\dabar@}%
  \@whiledim\count@\wd4<\dimen@\do{%
	\advance\count@\@ne
	\expandafter\def\expandafter\da@bars\expandafter{%
	  \da@bars
	  \dabar@ 
	}%
  }%
  \mathrel{#3}%
  \mathrel{%
	\mathop{\da@bars}\limits
	\ifx\\#1\\%
	\else
	  _{\copy0}%
	\fi
	\ifx\\#2\\%
	\else
	  ^{\copy2}%
	\fi
  }%
  \mathrel{#4}%
}
\makeatother

\renewcommand*{\mod}{\operatorname{mod\,}}
\newcommand*{\ZZ}{\mathbb{Z}}

\newcommand*{\CC}{\mathbb{C}}
\newcommand*{\PP}{\mathbb{P}}
\newcommand*{\FF}{\mathbb{F}}
\newcommand*{\Sing}{\op{Sing}}
\renewcommand*{\emptyset}{\varnothing}

\let\op\operatorname
\title{Maximally nodal sextic surfaces and linear determinantal representations}
\author{Yonghwa Cho}

\address{Department of Mathematics, Gyeongsang National University, 501 Jinju-daero, Jinju-si, Gyeongsangnam-do, 52828, Republic of Korea.}
\email{yhcho@gnu.ac.kr}

\thanks{This work was supported by Global - Learning \& Academic research institution for Master’s{}\textperiodcentered{}PhD students, and Postdocs(G-LAMP) Program of the National Research Foundation of Korea(NRF) grant funded by the Ministry of Education(No. RS-2023-00301974).}

\date{\today}

\begin{document}
	\setlength{\intextsep}{0.5\baselineskip plus 2pt minus 2pt}
	\abovedisplayshortskip=0.5\baselineskip%
	\belowdisplayshortskip=0.5\baselineskip%
	\abovedisplayskip=0.4\baselineskip%
	\belowdisplayskip=0.4\baselineskip%
\begin{abstract}
    We prove that every maximally nodal sextic surface\,(with $65$ nodes) $X \subset \PP_\CC^3$ contains a symmetric half-even set of nodes of cardinality $35$. It follows that the associated half-quadratic sheaf is the cokernel of a symmetric $6 \times 6$ matrix of linear forms, yielding a linear determinantal representation of $X$. In particular, after a suitable Serre twist, the half-quadratic sheaf is an Ulrich sheaf of rank $1$. As an example, we exhibit an explicit $6 \times 6$ matrix of linear forms whose determinant defines the Barth sextic surface.
\end{abstract}

\maketitle
\section{Introduction}

Let $X$ be a complex projective surface in $\PP^3$ defined by a homogeneous polynomial $F \in \CC[x_0,x_1,x_2,x_3]$. We call $X$ a \emph{nodal surface} if each singular point of $X$ is a node; or equivalently, at each singular point $P \in X$ we have
\[
    \op{rank} \left( \frac{\partial^2 F}{\partial x_i \partial x_j} \right)_P = 3.
\]
Once the degree $d = \deg X$ is fixed, the number of nodes on $X$ is bounded. Let $\mu(d)$ denote the maximal number of nodes that a nodal surface of degree $d$ can contain. By Miyaoka\,\cite{Miyaoka}, we have
\begin{equation}\label{eq: Miyaoka ineq}
    \mu(d) \leq \frac 49 d(d-1)^2.
\end{equation}
The precise values of $\mu(d)$ are known only for $d \leq 6$.\,\cite{Barth,BeauvilleNodal,JaffeRuberman,Togliatti,Wahl}
\[
    \begin{array}{c|c|c|c|c|c}
         d&  2& 3& 4& 5& 6\\ \hline
         \mu(d) & 1 & 4 & 16 & 31 & 65
    \end{array}
\]
In particular, the inequality \eqref{eq: Miyaoka ineq} is not optimal except when $d=4$. It is a long-standing question to determine $\mu(d)$, dating back to Kummer in the 19th century.

On the other hand, the classification of nodal surfaces with a large number of nodes is also an intriguing problem. In the collection \cite{NSVar} of papers, the authors study the \emph{nodal Severl variety} $\mathcal F_n(d,\nu)$ which parametrizes hypersurfaces of degree $d$ in $\PP^{n+1}$ containing precisely $\nu$ nodes as singularities. In \cite{CatNSVarCh1}, Catanese gave a thorough description of the varieties $\mathcal F_2(4,\nu)$ for $\nu \leq 16$; see \cite[Theorem~1.37]{CatNSVarCh1}. In particular, the nodal Severi variety $\mathcal F_2(4,16)$ of Kummer quartic surfaces is smooth and irreducible. The same properties hold for the variety $\mathcal F_2(5,31)$ of Togliatti quintic surfaces; see \cite[Theorem~2.29]{CatChoCoughFrap}. In the proof, the authors realize Togliatti quintic surfaces as the discriminants of linear projections of nodal cubic $4$-folds. This framework further developed in \cite{CatChoKie}, where maximally nodal sextic surfaces are described as the discriminants of linear projections of nodal cubic $5$-folds. However, the sextic case is considerably more complicated than the lower degree cases, and the irreducibility and smoothness of $\mathcal F_2(6,65)$ remain open.

In this article, we focus on the sextic surfaces with the maximal number\,(${=65}$) of nodes. Up to projective equivalence there exists a three-dimensional family of sextic surfaces with the maximal number of nodes\,\cite{Dimca, Kloosterman}. The first (and perhaps most famous) example is due to Barth, which is defined by the sextic equation
\begin{equation}\label{eq: Barth sextic}
    F(x_0,x_1,x_2,x_3) = (\tau^2x_1^2 - x_2^2)(\tau^2 x_2^2 - x_3^2)(\tau^2 x_3^2 - x_1^2) - \frac {2\tau+1}{4}x_0^2 (x_0^2 - x_1^2-x_2^2-x_3^2)^2,
\end{equation}
where $\tau = \frac12(1+\sqrt 5)$. Apart from Barth's example, maximally nodal sextic surfaces do not seem to have been studied extensively. In recent work with collaborators\,\cite{CatChoKie}, we study the extended code of a maximally nodal sextic surface and prove that they are all isomorphic. In \cite{CatChoKie}, we also prove that every half-even set of cardinality $31$ is symmetric. Consequently, we show that the defining equation of such a surface can be expressed as a determinant of a symmetric matrix of homogeneous forms of diagonal degrees $(1,1,1,3)$. 
For details of the previous work and related preliminaries, see Sections~\ref{sec: half-even sets} and \ref{sec: extended code of 65-nodal sextic}.

The main result of this article concerns linear determinantal representations. We prove that for every maximally nodal sextic surface $X \subset \PP^3$, there exists a symmetric half-even set of cardinality $35$. Given the material in Sections~\ref{sec: half-even sets} and \ref{sec: extended code of 65-nodal sextic}, we may restate the main result as follows.
\begin{theorem}\label{thm: main}
    Let $X \subset \PP^3$ be a sextic surface whose singular set consists of 65 nodes. Then, there exists a coherent sheaf $\mathcal F$ on $X$ which fits into the short exact sequence
    \[
        0 \to \mathcal O_{\PP^3}(-4)^{\oplus6} \xto{\,A\,} \mathcal O_{\PP^3}(-3)^{\oplus 6} \to \mathcal F \to 0,
    \]
    where $A$ is a symmetric matrix of linear forms on $\PP^3$. In particular, the surface $X = \op{Supp}\mathcal F$ is defined by the equation $(\det A=0)$.
\end{theorem}

\begin{remark}
Compared to the previous results in \cite{CatChoKie}, Theorem~\ref{thm: main} reveals another interesting aspect; the sheaf $\mathcal F(3)$ is an Ulrich sheaf of rank $1$ on $X$. 

Among the various well-known equivalent definitions of Ulrich sheaves, we take the one from \cite[Propposition~2.1(c)]{EisenbudSchreyer}, namely, a sheaf $\mathcal U$ on a projective variety $Y^n \subseteq \PP := \PP^{n+c}$ is Ulrich if it admits a linear free resolution of length $c$:
\[
    0 \to \mathcal O_\PP(-c)^{\oplus \beta_c} \to \mathcal O_\PP(-c+1)^{\oplus \beta_{c-1}} \to \dots \to \mathcal O_\PP(-1)^{\oplus \beta_1} \to \mathcal O_\PP^{\oplus \beta_0} \to \mathcal U \to 0.
\]
It is interesting and widely open question to ask whether every (nonsingular) projective variety admits an Ulrich sheaf. On the other hand, the existence of Ulrich line bundles is often not expected. By Noether-Lefschetz theorem, a very general sextic surface has Picard number one, thus by \cite[Section~4]{BeauvilleUlrich} such surfaces do not admit Ulrich line bundles. Theorem~\ref{thm: main} says that, unlike the very general case, maximally nodal sextic surfaces always admit Ulrich sheaf of rank $1$.

\end{remark}


\bigskip
\section{Half-even sets of nodes}\label{sec: half-even sets}
Thoughout this section, we fix the following notations:
Let $X = X_d \subset \PP^3$ be a nodal surface of degree $d$, and let $\Sing X$ be the singular set of $X$. For a minimal resolution $\pi \colon \tilde X \to X$ and a point $P \in \Sing X$, we denote by $E_P \in \Pic \tilde X$ the $(-2)$-curve over $P$. For a subset $\mathcal N \subset \Sing X$, $E_\mathcal N = \sum_{P \in \mathcal N} E_P$. Let $\tilde H \in \lvert \pi^*\mathcal O_{X}(1) \rvert$ be the pullback of a hyperplane section.

\smallskip
\subsection{Symmetric half-even sets of nodes}

\begin{definition}
A subset $\mathcal N \subset \Sing X$ is said to be \emph{even} if $E_\mathcal N \in 2\Pic \tilde X$. If $d$ is even and $\tilde H + E_\mathcal N \in 2 \Pic \tilde X$, we say that the subset $\mathcal N$ is \emph{half-even}.
\end{definition}

Note that if $\mathcal N$ is a half-even set, then
\[
    (\tilde H + E_\mathcal N)^2 = d + (-2)\lvert \mathcal N\rvert \equiv 0\ (\mod 4),
\]
hence the degree of $X$ must be even. 

For the rest of this section, we assume that $X$ contains a half-even set $\mathcal N$ of nodes. Let $\tilde p \colon \tilde S \to \tilde X$ be a double cover branched over $\tilde H + E_\mathcal N$. We have an eigensheaf decomposition
\[
\tilde p_* \mathcal O_{\tilde S} = \mathcal O_{\tilde X} \oplus \mathcal O_{\tilde X}(-L), \text{ where } L = \tfrac 12 (\tilde H + E_\mathcal N).
\]
The following theorem is one of the key ingredients of this paper.

\begin{theorem}[{\cite[Theorem~0.3 and Proposition~3.1]{CasCat}}]\label{thm: Casnati-Catanese}
    The sheaf $\mathcal F:= \pi_* \mathcal O_{\tilde X}(-L)$ admits a two term locally free resolution on $\PP^3$. More precisely, there exists a locally free sheaf $\mathcal E$ on $\PP^3$ which fits into the short exact sequence
    \begin{equation}\label{eq: Casnati-Catanese resolution}
        0 \to \mathcal E^\vee (-d-1) \xto{\varphi} \mathcal E \to \mathcal F \to 0.
    \end{equation}
    Moreover, the map $\varphi$ is symmetric.
\end{theorem}

\begin{definition}
    The half-even set $\mathcal N$ is \emph{symmetric} if $\mathcal E$ in \eqref{eq: Casnati-Catanese resolution} splits into a direct sum of line bundles. If $\mathcal N$ is symmetric, then the map $\varphi$ is given by a matrix of homogeneous forms. The diagonal degrees $(d_1,\dots,d_r)$\,($d_1\leq d_2\leq \dots \leq d_r$) of $\varphi$ is called the \emph{type} of $\mathcal N$.
\end{definition}

If $\mathcal N$ is symmetric half-even set of type $(d_1,\dots,d_r)$, then
\begin{equation}\label{eq: symmetric splitting type}
    \mathcal E \simeq \bigoplus_{i=1}^r \mathcal O_{\PP^3}\Bigl( -\frac{1}{2}(d+1-d_i) \Bigr).
\end{equation}
See \cite[Theorem~2.16]{CatBabbage}.

In \cite{CatChoKie}, we proved that if $X$ is a sextic surface with $65$ nodes, then every half-even set of nodes of cardinality $31$ is symmetric of type $(1,1,1,3)$, realizing $\mathcal F$ as the cokernel of a symmetric matrix
\[
    \mathcal O_{\PP^3}(-4)^{\oplus 3} \oplus \mathcal O_{\PP^3}(-5) \xto{\ \varphi\ } \mathcal O_{\PP^3}(-3)^{\oplus 3} \oplus \mathcal O_{\PP^3}(-2).
\]

By Horrocks splitting criterion, $\mathcal N$ is symmetric if and only if $H^p(\PP^3, \mathcal E(k))=0$ for each $p=1,2$ and $k$. We will show that the symmetric criterion reduces to the vanishing of cohomology groups of $\mathcal O_{\tilde X}(m\tilde H - L )$ for some $m$.

\begin{proposition}[{cf. \cite[Proposition~2.18]{CatBabbage} and \cite[Section~4.1]{CatChoKie}}]\label{prop: aCM criterion}
    Suppose $d \geq 4$. The half-even set $\mathcal N$ is symmetric if and only if
    \[
        h^1( \tilde X ,\, \mathcal O_{\tilde X}(  m\tilde H - L) ) = 0\ \ \text{for }\ \tfrac d2 -1 \leq m \leq d-3.
    \]
\end{proposition}
\begin{proof}
    By iii) in \cite[p.~242]{CasCat} and Serre duality, we have $h^2(\PP^3,\, \mathcal E(m)) = 0 = h^1(\PP^3,\, \mathcal E^\vee(\ell))$ for all $m$ and $\ell$. Thus, $h^1(\PP^3,\, \mathcal E(m)) = 0$ if and only if $h^1(X,\, \mathcal F(m)) = h^1(\tilde X,\, \mathcal O_{\tilde X}(m \tilde H - L))=0$.

    As above, let $\tilde p \colon \tilde S \to \tilde X$ be a double cover branched over $\tilde H+E_\mathcal N$. By Riemann-Hurwitz formula, we have $K_{\tilde S} \sim \tilde p^*( K_{\tilde X} + L) \sim (d-4)\, \tilde p^* \tilde H + \tilde p^* L$. Since $ \tilde p^* \tilde H$ is nef and big, Kawamata-Viewheg vanishing reads $h^1(\tilde S,\, \mathcal O_{\tilde S}( (d-4+m')\tilde p^* \tilde H + \tilde p^* L)) = 0$ for any positive integer $m'$. By projection formula,
    \begin{align*}
        \tilde p_* \mathcal O_{\tilde S}( (d-4 +m')\tilde p^* \tilde H + \tilde p^* L) %
        &= \tilde p_* \mathcal O_{\tilde S} \otimes \mathcal O_{\tilde X}( (d-4+m')\tilde H + L) \\
        &= \mathcal O_{\tilde X}( (d-4+m')\tilde H + L) \oplus \mathcal O_{\tilde X}( (d-4+m')\tilde H ).
    \end{align*}
    This in particular shows that
    \begin{equation}\label{eq: Kawamata-Viewheg vanishing}
        h^1(\tilde X,\, \mathcal O_{\tilde X}( m\tilde H + L)) = 0\ \ \text{for any integer } m > d-4.
    \end{equation}

    Let $v$ be an odd integer. From the short exact sequence
    \[
        0 \to \mathcal O_{\tilde X}\Bigl( \tfrac 12 ( v\tilde H - E_\mathcal N)\Bigr) \to \mathcal O_{\tilde X} \Bigl( \tfrac 12 ( v \tilde H + E_\mathcal N )\Bigr) \to \bigoplus_{P \in \mathcal N } \mathcal O_{E_P}(-1) \to 0,
    \]
    it follows that
    \begin{equation}\label{eq: half-system isom}
        H^p\Bigl( \tilde X, \, \mathcal O_{\tilde X}\Bigl( \tfrac 12 (v \tilde H - E_\mathcal N)\Bigr) \Bigr) \simeq %
        H^p\Bigl( \tilde X, \, \mathcal O_{\tilde X}\Bigl( \tfrac 12 (v \tilde H + E_\mathcal N)\Bigr) \Bigr)
    \end{equation}
    for each $p$. By Serre duality, we have
    \begin{align}
        h^p\Bigl( \tilde X, \, \mathcal O_{\tilde X}\Bigl( \tfrac 12 (v \tilde H - E_\mathcal N)\Bigr) \Bigr) %
        &= h^{2-p}\Bigl( \tilde X, \, \mathcal O_{\tilde X}\Bigl( K_{\tilde X} - \tfrac 12 (v \tilde H - E_\mathcal N)\Bigr) \Bigr) \nonumber\\%
        &= h^{2-p}\Bigl( \tilde X, \, \mathcal O_{\tilde X}\Bigl( \tfrac 12 \bigl( (2d-8-v) \tilde H - E_{\mathcal N} \bigr )\Bigr) \Bigr). \label{eq: Serre duality for quadratic sheaf}
    \end{align}
    Combining \eqref{eq: Serre duality for quadratic sheaf} with \eqref{eq: Kawamata-Viewheg vanishing} and \eqref{eq: half-system isom}, we deduce that
    \[
        h^1\Bigl(\tilde S,\, \mathcal O_{\tilde S}\Bigl( \tfrac1 2 ( v \tilde H - E_\mathcal N)\Bigr) \Bigr) = 0\ \ \text{if either $v \geq 2d-5$ or $v \leq -3$}.
    \]

    Consequently, $\mathcal E$ is splits into a direct sum of line bundles if and only if
    \[
        h^1( \tilde X,\, \mathcal O_{\tilde X}(k\tilde H - L)) = h^1 \Bigl(\tilde X, \, \mathcal O_{\tilde X}\Bigl( \tfrac 12 \bigl( (2k-1) \tilde H - E_\mathcal N \bigr) \Bigr) \Bigr) =0
    \]
    whenever $-1 \leq 2k-1 \leq 2d-7$ holds. By Serre duality, this is equivalent to have $h^1(\tilde X,\, \mathcal O_{\tilde X}(k \tilde H - L)) = 0$ for $k= \tfrac d2 - 1, \dots, d-3$.
\end{proof}

\smallskip
\subsection{Minimal half-even sets and nonreduced plane sections}
In this subsection, we collect some results from \cite{Endrass} that will be used throughout the paper. For details, see \cite{Endrass}.

Let $\mathcal N \subseteq \Sing X$ be a half-even set of nodes. Let $v$ be an odd integer such that $h^0\Bigl( \tilde X, \mathcal O_{\tilde X} \Bigl( \tfrac 12( v \tilde H - E_\mathcal N ) \Bigr) \Bigr)>0$. For a nonzero section $s \in H^0\Bigl( \tilde X, \mathcal O_{\tilde X} \Bigl( \tfrac 12( v \tilde H - E_\mathcal N ) \Bigr) \Bigr)$, let us denote by $\tilde D = \div(s) $ the divisor of $s$. Since $2\tilde D$ belongs to the linear system $\lvert \mathcal O_{\tilde X}(v\tilde H - E_\mathcal N )\rvert$, the push-forward $2D = \pi_*(2\tilde D)$ can be viewed as a member of $\lvert \mathcal O_X(v) \rvert$, namely, there exists a surface $V \subset \PP^3$ of degree $v$ such that $V.X = 2D$. As $(\tilde D.E_P)=1$ for each $P \in \mathcal N$, we find that $\mathcal N$ is precisely the set of singular points at which the curve $D$ has odd multiplicity.

\begin{definition}\label{def: half-even set cut out by a surface}
    A half-even set $\mathcal N$ is said to be \emph{cut out by a surface $V$} if $V.X = 2D$ and $D$ satisfies the following property: $P \in \Sing X$ belongs to $\mathcal N$ if and only if the multiplicity of $D$ at $P$ is odd.
\end{definition}

Assume that $v < d = \deg X$. From the short exact sequence
\[
    0 \to \mathcal O_{\PP^3}( v-d)  \to \mathcal O_{\PP^3}(v) \to \mathcal O_X(v) \to 0,
\]
it follows that $V$ is uniquely determined. A key idea of \cite{Endrass} is to study the \emph{quadratic system} $\mathsf Q_\mathcal N(v) \subset \lvert \mathcal O_{\PP^3}(v) \rvert$ on $\PP^3$ rather than the linear system
\[
\mathsf L_\mathcal N(v) := \Bigl\lvert \mathcal O_{\tilde X}\Bigl(\tfrac 12 ( v \tilde H - E_\mathcal N)\Bigr)\Bigr\vert.
\]
Let $s_0,\dots,s_n$ be a basis for $H^0\Bigl( \tilde X, \mathcal O_{\tilde X} \Bigl( \tfrac 12( v \tilde H - E_\mathcal N ) \Bigr) \Bigr)$. For $0 \leq i \leq j \leq n$, the section $s_is_j \in H^0( \tilde X,\, \mathcal O_{\tilde X}(v\tilde H- E_\mathcal N) )$ corresponds to $g_{ij} \in H^0(\PP^3, \, \mathcal O_{\PP^3}(v))$. Then, we have
\[
    \mathsf Q_{\mathcal N}(v) \simeq \Bigl\{ [\lambda_0,\dots,\lambda_n] \in \PP^n : g(\lambda):= \sum_{i=0}^n \lambda_i^2 g_{ii} + 2 \sum_{i<j} \lambda_i \lambda_j g_{ij} = 0\Bigr\},
\]
where the isomorphism induced by the embedding $\PP^n \to \PP^{v+3 \choose 3}$, which maps $[\lambda_0,\dots,\lambda_n]$ to the monomial coefficients of the homogeneous polynomial $g(\lambda)$ of degree $v$. This justifies the name \emph{quadratic system}. Note that the members of $\mathsf Q_{\mathcal N}(v)$ are identified with homogeneous polynomials of degree $v$, which are more explicit than those of $\mathsf L_{\mathcal N}(v)$.

The greatest common divisor of $\{ g(\lambda) : \lambda \in \mathsf Q_\mathcal N(v) \}$ defines the base locus $\mathsf B_\mathcal N(v)$ of $\mathsf Q_\mathcal N(v)$. We have the decomposition
\[
    \mathsf Q_\mathcal N(v) = \mathsf B_\mathcal N(v) + \mathsf F_\mathcal N(v),
\]
where $\mathsf F_\mathcal N(v)$ is the \emph{free part} of $\mathsf Q_\mathcal N(v)$. Note that $\mathsf F_\mathcal N(v)$ is base point free in codimension one; its base locus (if nonempty) consists of points and curves.

\begin{definition}
Let $\mathcal N$ be a half-even set of nodes, and let $v$ be an odd integer such that $\mathsf Q_\mathcal N(v) \neq \emptyset$. Then, $\mathcal N$ is said to be
\begin{enumerate}
    \item \emph{stable in degree $v$} if $\mathsf F_\mathcal N(v) = \emptyset$;
    \item \emph{semistable in degree $v$} if $\mathsf F_\mathcal N(v)$ is either empty or base point free;
    \item \emph{unstable in degree $v$} if $\mathsf F_\mathcal N(v)$ contains a base point.
\end{enumerate}
\end{definition}

A member of the form $2W$ is called a \emph{square}\,(as its defining equation is a square). From Definition~\ref{def: half-even set cut out by a surface}, one finds that $2W$ cuts out $\emptyset$. In particular, if $\mathsf F_\mathcal N(v)$ contains a square, then it contains all the squares of the same degree. Thus, $\mathsf F_\mathcal N(v)$ is base point free and $\mathsf B_\mathcal N(v)$ cuts out $\mathcal N$. The converse is also true according to the following proposition.
\begin{proposition}[{\cite[Proposition~3.6]{Endrass}}]
    Let $\mathcal N$ is a half-even set. Assume $\mathsf Q_\mathcal N(v) \neq \emptyset$. Then,
    \begin{enumerate}
        \item $\mathcal N$ is stable in degree $v$ if and only if $\dim \mathsf L_\mathcal N(v) = 0$;
        \item $\mathcal N$ is semistable in degree $v$ if and only if either $\mathsf F_\mathcal N(v)=\emptyset$ or $\mathsf F_\mathcal N(v)$ contains a square.
    \end{enumerate}
\end{proposition}

Suppose $\mathcal N$ is semistable in degree $v$. If $\deg \mathsf B_\mathcal N(v) = b$, then $\mathsf B_\mathcal N(b) = \mathsf B_\mathcal N (v)$ and $\mathsf F_\mathcal N(b) = \emptyset$, thus $\mathcal N$ is stable in degree $b$. The system $\mathsf Q_\mathcal N(v)$ behaves much better under the semistability assumption. The following proposition provides a useful criterion for semistability.

\begin{proposition}[{\cite[Corollary~3.9]{Endrass}}]\label{prop: semistable criteria}
    Let $\mathcal N$ be a half-even set of nodes.
    \begin{enumerate}
        \item If $\mathcal N$ is unstable in degree $v = \frac d2$, then $\lvert \mathcal N \rvert = (\frac d2)^3$.
        \item If $\mathcal N$ is unstable in degree $v = \frac 12 (d+1)$, then $\lvert \mathcal N\rvert \geq \frac 18 (d-1)d(d+1)$.
        \item If $\mathcal N$ is unstable in degree $v = \frac d2+1$, then $\lvert \mathcal N\rvert \geq \frac 18 (d-2) d(d+2)$.
    \end{enumerate}
\end{proposition}

We close this subsection by restating the half-even part of the main theorem of \cite{Endrass}.
\begin{theorem}[{\cite[Theorem~1.10]{Endrass}}]\label{thm: Endrass min half-even set}
    Assume $d \in \{2,4,6,8\}$. If $\mathcal N$ is a nonempty half-even set, then $\lvert \mathcal N \rvert \geq \frac 12 d(d-1)$ and the equality holds if and only if $\mathcal N$ is stable in degree $1$, or equivalently, $\mathcal N$ is cut out by a plane. 
\end{theorem}

\medskip
\section{The extended code of maximally nodal sextic surfaces}\label{sec: extended code of 65-nodal sextic}
\newcommand*{\symdiff}{\mathbin\triangle}%
Let $\mathcal N_1,\, \mathcal N_2$ be even sets of nodes. Their symmetric difference
\[
    \mathcal N_1 \symdiff \mathcal N_2 := (\mathcal N_1 \cup \mathcal N_2) \setminus (\mathcal N_1 \cap \mathcal N_2)
\]
is again an even set of nodes since $E_{\mathcal N_1 \symdiff \mathcal N_2} =  E_{\mathcal N_1} + E_{\mathcal N_2} - 2 E_{\mathcal N_1 \cap \mathcal N_2} \in 2 \, \Pic \tilde X$. Similarly, if $\mathcal N_1$ is half-even and $\mathcal N_2$ is even, then $\mathcal N_1 \symdiff \mathcal N_2$ is half-even. The symmetric difference of two half-even sets is even. From these observations we can naturally associate an $\FF_2$-vector space which consists of (half-)even sets.

\begin{definition}
Let $\Sing X = \{P_1,\dots,P_\nu\}$, and let $\mathcal V$ be an $\FF_2$-vector space of dimension $\nu+1$. Let $\{ e_0, e_1,e_2,\dots,e_{\nu} \}$ be an $\FF_2$-basis for $\mathcal V$. To each (half-)even set $\mathcal N$ we associate a \emph{characteristic vector}
\[
    w_\mathcal N = \left\{
    \begin{array}{rl}
        e_0 + \sum_{P_i \in \mathcal N} e_i & \text{ if $\mathcal N$ is half-even}\\
        \sum_{P_i \in \mathcal N} e_i & \text{ if $\mathcal N$ is even}.
    \end{array}
    \right.
\]
The subspace $\mathcal K$ (resp. $\bar{\mathcal K})$ of the characteristic vectors of even sets (resp. even and half-even sets) is called the \emph{strict code} (resp. \emph{extended code}) of $X$.

An element $w \in \bar{\mathcal K}$ is called the \emph{codeword}, and is written as the binary string
\[
    w = (\,w^0 \,w^1\, w^2\, \dots\, w^\nu\,),
\]
where $w = \sum_{i=0}^\nu w^i e_i$ with $w^i \in \{0,1\}$. The weight $\lvert w \rvert$ of $w$ is the integer $\sum_{i=0}^\nu w^i$, where the sum is taken in $\ZZ$.
\end{definition}

We note that if $\mathcal N$ is half-even, then $\lvert w_\mathcal N \rvert = 1 + \lvert \mathcal N \rvert$ as $w^0=1$. One of the main results of \cite{CatChoKie} and \cite{Kurz} identifies the extended code of the maximally nodal sextic surfaces.
\begin{theorem}[{\cite[Section~4.3]{CatChoKie} and \cite[Theorem~5.3]{Kurz}}]\label{thm: extended code of 65-nodal sextic}
Let $X$ be a nodal sextic surface with $65$ nodes, and let $\bar{\mathcal K} \subseteq \FF_2^{66}$ be an extended code. Then, $\bar{\mathcal K}$ is isomorphic to the vector space spanned by the following basis:\footnote{We put some spaces and vertical lines for better readability.}
\[
    \scalebox{.75}{$
    \left(
    \begin{array}{c @{\,|\,} c @{\,} c @{\,} c @{\,} c @{\,|\,} c @{\,} c @{\,} c@{\,} c @{\,|\,} c@{\,} c@{\,} c@{\,} c}
        1 & 0000 & 0000&0110&001&00000110&10000001&00011000&111111&10011111&0011&1100&1010 \\
        0 & 1000 & 0001&0001&010&00001110&00001011&11000001&100110&11000100&1011&0001&0100 \\
        0 & 0100 & 0001&0111&111&00001011&11000010&11101001&100110&00100000&1101&1011&0111 \\
        0 & 0010 & 0001&0010&001&00000111&10000101&00011010&101010&10011000&0010&1000&1011 \\
        0 & 0001 & 0001&0100&100&00001101&00010110&00001110&010101&00010101&0001&0111&0010 \\
        0 & 0000 & 1001&0011&110&00000101&01100011&10110001&001111&11010111&0110&0101&1100 \\
        0 & 0000 & 0101&0110&011&00001100&10001011&01011001&111111&11011110&0000&1100&1100 \\
        0 & 0000 & 0011&0101&101&00000110&10110010&00101011&001111&10011111&0011&1100&1010 \\
        0 & 0000 & 0000&1111&000&00001111&10100101&00111100&111111&10010110&0101&0101&1010 \\
        0 & 0000 & 0000&0000&000&10001000&00100010&01000100&111111&11101110&1010&1010&0101 \\
        0 & 0000 & 0000&0000&000&01000100&00010001&10001000&111111&11011101&1010&0101&1010 \\
        0 & 0000 & 0000&0000&000&00100010&10001000&00010001&111111&10111011&0101&1010&1010 \\
        0 & 0000 & 0000&0000&000&00010001&01000100&00100010&111111&01110111&0101&0101&0101 \\
    \end{array}
    \right)\raisebox{-6\baselineskip}{.}
    $}
\]
\end{theorem}

\medskip
The minimum weight of nonzero codewords in $\bar{\mathcal K}$ is $16$, and there are $26$ such codewords. These codewords play an important role in this paper. Table~\ref{tbl: minimal half-even sets} is the list of the codewords of weight $16$ in the code $\bar{\mathcal K}$ of Theorem~\ref{thm: extended code of 65-nodal sextic}. Note that by Theorem~\ref{thm: Endrass min half-even set}, each half-even set in Table~\ref{tbl: minimal half-even sets} is cut out by a plane.
\begin{table}[h!]
\scalebox{0.75}{
\begin{tabular}{r@{(\,}c @{\,|\,} c @{\,} c @{\,} c @{\,} c @{\,|\,} c @{\,} c @{\,} c@{\,} c @{\,|\,} c@{\,} c@{\,} c@{\,} c@{\,)}}
    \hline
    $w_1={}$&1&0010&0001&0100&000&00010000&01000000&00100000&101010&01110000&0100&0001&0100\\
    $w_2={}$&1&1000&0100&0001&000&01000000&00010000&00001000&011001&01011000&0010&0100&1000\\
    $w_3={}$&1&0001&0010&1000&000&00100000&00001000&00010000&100101&00111000&0001&1000&0010\\
    $w_4={}$&1&0100&1000&0010&000&00001000&00100000&01000000&010110&01101000&1000&0010&0001\\
    $w_5={}$&1&0010&0001&0100&000&00000001&00000100&00000010&010101&00000111&0001&0100&0001\\
    $w_6={}$&1&0000&1010&0000&010&00010100&00010100&10100000&000000&10100000&0011&0000&1001\\
    $w_7={}$&1&1100&0000&0000&100&00100001&11000000&00100001&000000&11000000&0000&1100&0011\\
    $w_8={}$&1&0000&0000&1001&001&10010000&01000010&01000010&000000&10010000&1001&0110&0000\\
    $w_9={}$&1&1000&0100&0001&000&00000100&00000001&10000000&100110&10000101&1000&0001&0010\\
    $w_{10}={}$&1&0000&1010&0000&010&01000001&01000001&00001010&000000&00001010&1100&0000&0110\\
    $w_{11}={}$&1&0000&0000&0110&001&01000010&10010000&10010000&000000&01000010&1001&1001&0000\\
    $w_{12}={}$&1&0011&0000&0000&100&00001100&00010010&00001100&000000&00010010&0000&0011&0011\\
    $w_{13}={}$&1&0001&0010&1000&000&00000010&10000000&00000001&011010&10000011&0100&0010&1000\\
    $w_{14}={}$&1&1100&0000&0000&100&00010010&00001100&00010010&000000&00001100&0000&0011&1100\\
    $w_{15}={}$&1&0000&0000&0110&001&00100100&00001001&00001001&000000&00100100&0110&0110&0000\\
    $w_{16}={}$&1&0000&0101&0000&010&10100000&10100000&00010100&000000&00010100&1100&0000&1001\\
    $w_{17}={}$&1&0100&1000&0010&000&10000000&00000010&00000100&101001&10000110&0010&1000&0100\\
    $w_{18}={}$&1&0000&0000&1001&001&00001001&00100100&00100100&000000&00001001&0110&1001&0000\\
    $w_{19}={}$&1&0011&0000&0000&100&11000000&00100001&11000000&000000&00100001&0000&1100&1100\\
    $w_{20}={}$&1&0000&0101&0000&010&00001010&00001010&01000001&000000&01000001&0011&0000&0110\\
    $w_{21}={}$&1&1001&1001&0000&001&10011001&10011001&00000000&000011&00000000&0000&0000&0000\\
    $w_{22}={}$&1&0110&0110&0000&001&01100110&01100110&00000000&000011&00000000&0000&0000&0000\\
    $w_{23}={}$&1&0000&1100&1100&100&00000000&11001100&11001100&110000&00000000&0000&0000&0000\\
    $w_{24}={}$&1&0000&0011&0011&100&00000000&00110011&00110011&110000&00000000&0000&0000&0000\\
    $w_{25}={}$&1&1010&0000&1010&010&10101010&00000000&10101010&001100&00000000&0000&0000&0000\\
    $w_{26}={}$&1&0101&0000&0101&010&01010101&00000000&01010101&001100&00000000&0000&0000&0000\\ \hline
\end{tabular}}
\caption{Minimal half-even set of nodes.}\label{tbl: minimal half-even sets}
\end{table}

\bigskip
\section{Proof of Theorem~\ref{thm: main}}
In this section we will prove Theorem~\ref{thm: main}. For the rest of this section $X$ is a nodal sextic surface with $65$ nodes. By \eqref{eq: symmetric splitting type}, it suffices to show that $X$ contains a symmetric half-even set $\mathcal N$ of type $(1,1,1,1,1,1)$. Using Riemann-Roch theorem, we deduce
\begin{equation}\label{eq: Riemann-Roch for quadratic sheaf}
    \chi( \mathcal F(m)) = 11 + \frac 34 (2m-1)(2m-5) - \frac {\lvert \mathcal N \rvert}{4},
\end{equation}
where $\mathcal F$ is the sheaf as in Theorem~\ref{thm: Casnati-Catanese}. If $\mathcal N$ was symmetric, then by the short exact sequence in Theorem~\ref{thm: main} we have $\chi(\mathcal F) = 6$, or equivalently, $\lvert \mathcal N \rvert = 35$.

\begin{proposition}[{\cite[Lemma~4.7]{CatChoKie}}]\label{prop: symmetry criterion}
    Suppose $\mathcal N$ is a half-even set of cardinality $35$. If $h^0(X,\, \mathcal F(2))=0$, then $\mathcal N$ is symmetric of type $(1,1,1,1,1,1)$.
\end{proposition}
\begin{proof}
    Applying Serre duality \eqref{eq: Serre duality for quadratic sheaf}, we have
    \begin{equation}\label{eq: Serre duality for degree 6}
        h^p(X,\, \mathcal F(2)) = h^p\bigl(\tilde X,\, \mathcal O_{\tilde X}\bigl(\tfrac 12(3\tilde H-E_\mathcal N)\bigr)\bigr) = h^{2-p}\bigl( \tilde X,\, \mathcal O_{\tilde X}\bigl(  \tfrac 12 ( \tilde H - E_\mathcal N ) \bigr) \bigr) = h^{2-p}(X,\,\mathcal F(1)).
    \end{equation}
    Hence, $h^0(X,\, \mathcal F(1))=h^2(X,\,\mathcal F(2))=0$ and $h^2(X,\, \mathcal F(1)) = h^0(X,\,\mathcal F(2))=0$. By \eqref{eq: Riemann-Roch for quadratic sheaf}, $\chi(\mathcal F(1)) = \chi(\mathcal F(2))=0$, thus by Proposition~\ref{prop: aCM criterion}, $\mathcal N$ is symmetric.

    If $\mathcal N$ is symmetric of type $(d_1,\dots,d_r)$, then
    \[
        0 \to \bigoplus_{i=1}^r \mathcal O_{\PP^3}\Bigl(\tfrac12(-7-d_i)\Bigr) \to \bigoplus_{i=1}^r \mathcal O_{\PP^3}\Bigl(\tfrac12(-7+d_i)\Bigr) \to \mathcal F \to 0.
    \]
    As $h^0(X,\,\mathcal F(2))=0$, $d_i \leq 1$ for each $i$. Invoking the inequalities in \cite[p.249]{CasCat}, we see that $r = 6$ and $d_i=1$ for all $i$.
\end{proof}

In order to complete the proof, we need to show that there exists a half-even set $\mathcal N$ of cardinality $35$ such that $h^0(X ,\, \mathcal F(2))=0$. First, we gather the consequences of the following assumption:
\smallskip
\begin{center}
    \hfill Let $\mathcal N$ be a half-even set $\mathcal N$ of cardinality $35$ satisfying $h^0(X,\, \mathcal F(2) ) \neq 0$. \hfill (\hypertarget{dagger}{$\dagger$})
\end{center}

\newcommand*{\refdagger}{\hyperlink{dagger}{{\normalfont($\dagger$)}}}
\begin{proposition}\label{prop: h^0(F(2))=1}
    Assume \refdagger. We have $h^0(X,\, \mathcal F(2)) = 1$.
\end{proposition}
\begin{proof}
    Since $H^0(X,\,\mathcal F(2)) \simeq H^0\bigl( \tilde X, \mathcal O_{\tilde X}\bigl( \frac12 (3\tilde H - E_\mathcal N) \bigr)\bigr) \neq 0$, by Proposition~\ref{prop: semistable criteria} we see that $\mathcal N$ is semistable in degree $3$. Thus, $\mathcal N$ is stable in degree $1$ or $3$. By Theorem~\ref{thm: Endrass min half-even set}, $\mathcal N$ cannot be cut out by a plane, thus $\mathcal N$ is stable in degree $3$. In particular, $h^0\bigl( \tilde X, \mathcal O_{\tilde X}\bigl( \frac12 (3\tilde H - E_\mathcal N) \bigr)\bigr) = 1$.
\end{proof}

\begin{proposition}\label{prop: resolution over P^2}
    Assume \refdagger, and let $H \subset \PP^3$ be a general plane. Then, $\mathcal F\big\vert_H$ fits into one of the following short exact sequences on $H$: either
    \begin{equation}\label{eq: resolution type 33}
        0 \to \mathcal O_H(-5)^{\oplus 2} \to \mathcal O_H(-2)^{\oplus 2} \to \mathcal F \big\vert_H \to 0,
    \end{equation}
    or
    \begin{equation}\label{eq: resolution type 1113}
        0 \to \mathcal O_H(-4)^{\oplus 3} \oplus \mathcal O_H(-5) \to \mathcal O_H(-3)^{\oplus 3} \oplus \mathcal O_H(-2) \oplus \to \mathcal F \big\vert_H \to 0.
    \end{equation}
\end{proposition}
\begin{proof}
    By \cite[p.249]{CasCat}, there exists a short exact sequence
    \[
        0 \to \bigoplus_i \mathcal O_H\Bigl( \frac{-7-d_i}{2} \Bigr) \to \bigoplus_i \mathcal O_H\Bigl( \frac{-7+d_i}{2} \Bigr) \to \mathcal F\big\vert_H \to 0.
    \]
    Using the inequalities in \cite[p.249]{CasCat}, one can show that $(d_i)_i$ is either $(1,5)$, $(3,3)$, $(1,1,1,3)$, or $(1,1,1,1,1,1)$\,(see \cite[Lemma~4.4]{CatChoKie}). The sequences $(3,3)$ and $(1,1,1,3)$ correspond to \eqref{eq: resolution type 33} and \eqref{eq: resolution type 1113}, respectively. If $(d_i)_i = (1,5)$, then
    \[
        0 \to \mathcal O_H(-6) \oplus \mathcal O_H(-4) \to \mathcal O_H(-1) \oplus \mathcal O_H(-3) \to \mathcal F\big\vert_H \to 0,
    \]
    thus $h^0(H, \, \mathcal F(2) \big\vert_H ) = 3$. By \eqref{eq: Serre duality for degree 6} and Proposition~\ref{prop: h^0(F(2))=1}, we have $h^1(X,\, \mathcal F(2)) = h^1(X,\, \mathcal F(1)) = 1$. Then, from the short exact sequence
    \[
        0 \to \mathcal F(1) \to \mathcal F(2) \to \mathcal F(2) \big\vert_H \to 0,
    \]
    we get $h^0(H, \, \mathcal F(2)\big\vert_H) \leq h^0(X,\, \mathcal F(2)) + h^1(X,\,\mathcal F(1)) = 2$, a contradiction. If $(d_i)_i = (1,1,1,1,1,1)$, then we have $h^0(H,\, \mathcal F(2)\big\vert_H) = 0$, thus $h^0(X,\,\mathcal F(1)) = h^0(X,\,\mathcal F(2))=1$. This contradicts the stability of $\mathcal N$ in degree $3$.
\end{proof}

\begin{proposition}\label{prop: to show h^1(F)=0}
    Assume \refdagger, and let $H\subset \PP^3$ be a general plane. Then, the map $\gamma \colon H^1(X,\,\mathcal F) \to H^1(H,\, \mathcal F\big\vert_H)$ is zero.
\end{proposition}
\begin{proof}
    If $h^1(X,\,\mathcal F)=0$, then there is nothing to prove. Assume $h^1(X,\,\mathcal F) \neq 0$. By Propposition~\ref{prop: resolution over P^2}, $h^0(H,\, \mathcal F(1)\big\vert_H)=0$. From the short exact sequence
    \[
        0 \to \mathcal F \to \mathcal F(1) \to \mathcal F(1)\big\vert_H \to 0,
    \]
    we have $h^1(X,\,\mathcal F) = h^1(X,\, \mathcal F(1)) = 1$. In particular, the map
    \[
        H^1(X,\, \mathcal F(1)) \to H^1(H,\, \mathcal F(1)\big\vert_H)
    \]
    is zero. Let $H' \subset \PP^3$ be a general plane. Then, the multiplication by the corresponding linear form yields the commutative diagram as follows.
\begin{center}
\begin{tikzpicture}
    \pgfmathsetmacro{\x}{2}
    \pgfmathsetmacro{\y}{1.5}
    \draw(0,0) node[anchor=center] (00) {$\mathcal F$};
    \draw(\x,0) node[anchor=center] (10) {$\mathcal F(1)$};
    \draw(0,-\y) node[anchor=center] (01) {$\mathcal F\big\vert_H$};
    \draw(\x,-\y) node[anchor=center] (11) {$\mathcal F(1)\big\vert_H$};
    \draw[->] (00)--(10) node[midway, above] {$\scriptstyle s_{H'}$};
    \draw[->] (01)--(11);
    \draw[->] (00)--(01);
    \draw[->] (10)--(11);
\end{tikzpicture}
\end{center}
Since the right vertical map induces the zero map between first cohomology groups, we have
\[
    \image \gamma \subseteq \ker \lambda_{H'},
\]
where $\lambda_{H'} \colon H^1(H,\,\mathcal F\big\vert_H) \to H^1(H,\, \mathcal F(1)\big\vert_H)$.

To prove that $\gamma$ is zero, it suffices to show
\begin{equation}\label{eq: vanishing interseciton of kernels}
    \bigcap_{\substack{H'\subset\PP^3 \\\text{general}}} \ker \lambda_{H'} = 0.      
\end{equation}
From Proposition~\ref{prop: resolution over P^2} it follows that there are two cases, namely, \eqref{eq: resolution type 33} and \eqref{eq: resolution type 1113} to consider. We start with the first case, where \eqref{eq: resolution type 33} holds. We have a commutative diagram
\begin{center}
\begin{tikzpicture}
    \pgfmathsetmacro{\x}{3}
    \pgfmathsetmacro{\y}{1.5}
    \draw(0,0) node[anchor=center] (00) {$0$};
    \draw($(00)+(0.66*\x,0)$) node[anchor=center] (10) {$\mathcal O_H(-4)^{\oplus 2}$};
    \draw($(10)+(\x,0)$) node[anchor=center] (20) {$\mathcal O_H(-1)^{\oplus 2}$};
    \draw($(20)+(\x,0)$) node[anchor=center] (30) {$\mathcal F(1)\big\vert_H$};
    \draw($(30)+(0.66*\x,0)$) node[anchor=center] (40) {$0$};

    \draw(0,\y) node[anchor=center] (01) {$0$};
    \draw($(01)+(0.66*\x,0)$) node[anchor=center] (11) {$\mathcal O_H(-5)^{\oplus 2}$};
    \draw($(11)+(\x,0)$) node[anchor=center] (21) {$\mathcal O_H(-2)^{\oplus 2}$};
    \draw($(21)+(\x,0)$) node[anchor=center] (31) {$\mathcal F\big\vert_H$};
    \draw($(31)+(0.66*\x,0)$) node[anchor=center] (41) {$0$};

    \foreach \i in {0,1,2,3}{
        \pgfmathtruncatemacro{\j}{\i+1}
        \draw[->] (\i0)--(\j0);
        \draw[->] (\i1)--(\j1);
    }
    \foreach \i in {1,2,3}{
        \draw[->] (\i1)--(\i0);
    }
\end{tikzpicture}
\end{center}
where the vertical maps are induced by the linear form defining $H'$. The map $\lambda_{H'}$ is identified with
\[
    H^2(H,\, \mathcal O_H(-5)^{\oplus 2}) \to H^2(H,\,\mathcal O_H(-4)^{\oplus 2}),
\]
or equivalently via Serre duality, $H^0(H,\, \mathcal O_H(2)^{\oplus 2})^* \to H^0(H,\, \mathcal O_H(1)^{\oplus 2})^*$. Since the map $\mathcal O_H(-5)^{\oplus 2} \to \mathcal O_H(-4)^{\oplus 2}$ is diagonal, the previous map is the direct sum of two copies of 
\[
    \kappa_{H'} \colon H^0(H,\, \mathcal O_H(2))^* \to H^0(H,\, \mathcal O_H(1))^*.
\]
In particular, we have $\ker \lambda_{H'} \simeq ( \ker \kappa_{H'} ) ^{\oplus 2}$, and \eqref{eq: vanishing interseciton of kernels} follows by Lemma~\ref{lem: intersection of kernels}.

Next, we need to consider when \eqref{eq: resolution type 1113} holds in Proposition~\ref{prop: resolution over P^2}. Resolving the map $\mathcal F\big\vert_H \to \mathcal F(1)\big\vert_H$ using \eqref{eq: resolution type 1113} as in the previous case, we obtain a commutative diagram
\begin{center}
\begin{tikzpicture}[scale=1]
    \pgfmathsetmacro{\x}{3}
    \pgfmathsetmacro{\y}{1.5}
    \pgfmathsetmacro{\nodescale}{0.9}
    
    \draw(0,\y) node[anchor=center] (01) {$0$};
    \draw($(01)+(0.66*\x,0)$) node[anchor=center, scale=\nodescale] (11) {$H^1(H,\, \mathcal F\big\vert_H)$};
    \draw($(11)+(1.5*\x,0)$) node[anchor=center, scale=\nodescale] (21) {$H^2(H,\, \mathcal O_H(-5)  \oplus \mathcal O_{H}(-4)^{\oplus 3})$};
    \draw($(21)+(1.5*\x,0)$) node[anchor=center, scale=\nodescale] (31) {$H^2(H,\, \mathcal O_H(-3)^{\oplus 3})$};
    \draw($(31)+(0.8*\x,0)$) node[anchor=center] (41) {$0$};    
    
    \draw(0,0) node[anchor=center] (00) {$0$};
    \draw($(11)+(0,-\y)$) node[anchor=center, scale=\nodescale] (10) {$H^1(H,\, \mathcal F(1)\big\vert_H)$};
    \draw($(21)+(0,-\y)$) node[anchor=center, scale=\nodescale] (20) {$H^2(H,\, \mathcal O_H(-4)  \oplus \mathcal O_{H}(-3)^{\oplus 3})$};
    \draw($(31)+(0,-\y)$) node[anchor=center] (30) {$0$};

    \foreach \i in {0,1,2}{
        \pgfmathtruncatemacro{\j}{\i+1}
        \draw[->] (\i0)--(\j0);
        \draw[->] (\i1)--(\j1);
    }
    \draw[->] (31)--(41);
    \draw[->] (11)--(10) node[midway, right]{$\scriptstyle \lambda_{H'}$};
    \draw[->] (21)--(20) node[midway, right]{$\scriptstyle \rho_{H'}$};
    \draw[->] (31)--(30);
\end{tikzpicture}
\end{center}%
We have $\ker \lambda_{H'} \subseteq \ker \rho_{H'}$ by Snake lemma. Using Serre duality, we identify the map $\rho_{H'}$ with the direct sum of $\kappa_{H'} \colon H^0(H,\, \mathcal O_H(2))^* \to H^0(H,\, \mathcal O_H(1))^*$ with three copies of
\[
    \kappa_{H'}'\colon H^0(H,\, \mathcal O_H(1))^* \to H^0(H,\, \mathcal O_H)^*.
\]
Thus, $\ker \rho_{H'} = \ker \kappa_{H'} \oplus (\ker \kappa_{H'}')^{\oplus 3}$, and by Lemma~\ref{lem: intersection of kernels}, we have
\[
     \bigcap_{\substack{H'\subset\PP^3 \\\text{general}}} \ker \rho_{H'}  = 0,
\]
which implies \eqref{eq: vanishing interseciton of kernels} as desired.
\end{proof}

\begin{lemma}\label{lem: intersection of kernels}
    Let $\ell \in H^0(\PP^2,\, \mathcal O_{\PP^2}(1))$ be a nonzero linear form, and let 
    \[
        \kappa_\ell \colon H^0(\PP^2,\, \mathcal O_{\PP^2}(2))^* \to H^0(\PP^2, \, \mathcal O_{\PP^2}(1))^*
    \]
    be the map induced by $\ell$. Then, for any dense subset $U \subset H^0(\PP^2,\,\mathcal O_{\PP^2}(1))$, we have $\bigcap_{\ell \in U} \ker \kappa_\ell = 0$. The same is true for $\kappa'_\ell \colon H^0(\PP^2,\, \mathcal O_{\PP^2}(1))^* \to H^0(\PP^2, \, \mathcal O_{\PP^2})^*$.
\end{lemma}
\begin{proof}
    Let $x_0,x_1,x_2$ be homogeneous coordinates on $\PP^2$, and let $\{ f_{\alpha\beta}\}_{0 \leq \alpha\leq \beta \leq 2}$ be a basis for $H^0(\PP^2,\, \mathcal O_{\PP^2}(2))^*$ determined by
    \[
        \langle f_{\alpha\beta}, \, x_i x_j\rangle = \left\{
        \begin{array}{ll}
            1 & \text{if }\{\alpha,\beta\} = \{i,j\} \\
            0 & \text{otherwise}.
        \end{array}
        \right.
    \]
    For notational convenience, we set $\varphi_{\alpha\beta} := \varphi_{\beta\alpha}$ for $\alpha > \beta$. Let $\ell = \ell_0 x_0 + \ell_1 x_1 + \ell_2 x_2$ and let $\varphi = \sum_{\alpha\leq \beta} \varphi_{\alpha\beta}f_{\alpha\beta} \in H^0(\PP^2,\, \mathcal O_{\PP^2}(2))^* $. Then, $\kappa_\ell(\varphi)$ is defined as follows: for a linear form $h = h_0 x_0 + h_1 x_1 + h_2 x_2$, $\kappa_\ell(\varphi)$ maps $h$ to
    \begin{align*}
        \langle \kappa_\ell(\varphi), h \rangle &= \langle \varphi , \ell h \rangle \\
        &= \sum_{\alpha\leq \beta} \sum_{i,j} \varphi_{\alpha\beta} h_i \ell_j \langle f_{\alpha\beta} , x^ix^j \rangle \\
        &= \sum_i h_i \bigl (\sum_j \varphi_{ij} \ell_j\bigr ).
    \end{align*}
    Thus, we have
    \begin{equation}\label{eq: iff condition in lemma}
        \varphi \in \ker \kappa_\ell \iff \sum_j \varphi_{ij} \ell_j=0\ \ \text{for each }i=0,1,2.
    \end{equation}
    If $\varphi \neq 0$, then the right hand side of \eqref{eq: iff condition in lemma} imposes nontrivial linear relations on $\{\ell_i\}$, thus there always exists $\ell' \in U$ such that $\varphi \not\in \ker \kappa_{\ell'}$. It follows that $\bigcap_{\ell \in U} \ker \kappa_\ell = 0$. The last statement can be proved analogously.
\end{proof}

\begin{corollary}\label{cor: key long exact seq}
    Assume \refdagger. For any plane $\bar H \subset \PP^2$, we have a long exact sequence of cohomology groups
    \[
        0 \to H^0(X,\, \mathcal F(2)) \to H^0(X,\,\mathcal F(3)) \to H^0(\bar H,\, \mathcal F(3)\big\vert_{\bar H}) \to H^1(X,\, \mathcal F(1)) \to 0.
    \]
    The dimension vector of the above sequence is $(1,6,6,1)$.
\end{corollary}
\begin{proof}
    Since the sequence in the statement consists of the first four terms in the long exact sequence induced by
    \[
        0 \to \mathcal F(2) \to \mathcal F(3) \to \mathcal F(3)\big\vert_{\bar H} \to 0,
    \]
    it suffices to prove that $h^1(X, \, \mathcal F(3)) = 0$. Let $H \subset \PP^2$ be a general plane. In the short exact sequence
    \[
        0 \to \mathcal F(-1) \to \mathcal F \to \mathcal F\big\vert_H \to 0,
    \]
    we have the exact sequence $H^1(X,\, \mathcal F(-1)) \to H^1(X,\,\mathcal F) \to H^1(H,\, \mathcal F\big\vert_H)$. By Proposition~\ref{prop: to show h^1(F)=0}, the map $H^1(X,\,\mathcal F(-1)) \to H^1(X,\, \mathcal F)$ is surjective. On the other hand, by \eqref{eq: Kawamata-Viewheg vanishing}, \eqref{eq: half-system isom} and \eqref{eq: Serre duality for quadratic sheaf}, we have
    \[
        h^1(X,\, \mathcal F(-1)) = h^1(X,\, \mathcal F(4)) = 0.
    \]
    This shows that $h^1(X,\,\mathcal F) = h^1(X,\, \mathcal F(3))=0$. The computation of the dimension vector follows from $\chi(\mathcal F(3)) = \chi(\mathcal F(3)\big\vert_{\bar H}) = 6$
\end{proof}

\begin{proposition}\label{prop: F(3) surjective}
    There exists a half-even set $\mathcal N$ of cardinality $35$ and a plane $\bar H \subset \PP^2$ such that the map
    \[
        H^0(X,\, \mathcal F(3)) \to H^0(X,\, \mathcal F(3) \big\vert_{\bar H})
    \]
    is surjective.
\end{proposition}
\begin{proof}
    Let $w = w_1 + w_2 + w_3 + w_5 + w_{17}$\,(see Table~\ref{tbl: minimal half-even sets}) be an element of the extended code. Expressed as a binary string, $w$ is
\[
    \begin{array}{c @{\,|\,} c @{\,} c @{\,} c @{\,} c @{\,|\,} c @{\,} c @{\,} c@{\,} c @{\,|\,} c@{\,} c@{\,} c@{\,} c}
        1&1101&1110&1011&000&11110001&01011110&00111110&101010&10010001&0100&0001&1011
    \end{array}.
\]
    The associated half-even set $\mathcal N$ has cardinality $35$. Let us denote by $\mathcal L_w$ the sheaf $\mathcal O_{\tilde X}(\frac 12( 5\tilde H - E_\mathcal N))$, and by $\mathcal L_{i}$ the sheaf $\mathcal O_{\tilde X}(\frac 12(\tilde H - E_{\mathcal N_i}))$, where $\mathcal N_i$ is the half-even set associated with $w_i$. By Theorem~\ref{thm: Endrass min half-even set}, each cohomology group $H^0(\tilde X,\, \mathcal L_i)$ is isomorphic to $\CC$. Let $\tilde s_i$ be a nonzero section of $H^0(\tilde X,\, \mathcal L_i)$. We regard the sheaf $\mathcal L_1 \otimes \mathcal L_2 \otimes \mathcal L_3 \otimes \mathcal L_5 \otimes \mathcal L_{17}$ as a subsheaf of $\mathcal L_w$ as follows. Let $D \in \Pic \tilde X $ be the effective divisor determined by
    \[
        2D = E_{\mathcal N_1} + E_{\mathcal N_2} + E_{\mathcal N_3} + E_{\mathcal N_5} + E_{\mathcal N_{17}} - E_{\mathcal N}.
    \]
    Note that $E_{\mathcal N}$ is the reduced sum of $(-2)$-curves in $E_{\mathcal N_1} + E_{\mathcal N_2} + E_{\mathcal N_3} + E_{\mathcal N_5} + E_{\mathcal N_{17}}$ that appear an odd number of times. Then we have the natural injective morphism of sheaves
    \begin{equation}\label{eq: finding a section of L_w}
        \mathcal L_1 \otimes \mathcal L_2 \otimes \mathcal L_3 \otimes \mathcal L_5 \otimes \mathcal L_{17} \to \mathcal L_1 \otimes \mathcal L_2 \otimes \mathcal L_3 \otimes \mathcal L_5 \otimes \mathcal L_{17} \otimes \mathcal O_{\tilde X}(D) = \mathcal L_w.
    \end{equation}
    Let $\tilde s_{1,2,3,5,17} \in H^0(\tilde X,\,\mathcal L_w)$ be the section which is obtained by mapping $\tilde s_1 \otimes \tilde s_2 \otimes \tilde s_3 \otimes \tilde s_5 \otimes \tilde s_{17}$ under the morphism above. Via the natural identification $H^0(\tilde X,\, \mathcal L_w) \simeq H^0(X,\, \mathcal F(3))$, we obtain a section $s_{1,2,3,5,17} \in H^0(X,\,\mathcal F(3))$. If we denote by $\mathcal F_i$ the sheaf $\pi_* \mathcal O_{\tilde X}\bigl(- \frac 12(\tilde H + E_{\mathcal N_i})\bigr)$, then by \cite{CasCat}, each $\mathcal F_i$ is reflexive sheaf of rank $1$ on $X$. We pick $s_i \in H^0(X,\,\mathcal F_i(1))$ the section corresponding to $\tilde s_i \in H^0(\tilde X,\, \mathcal L_i)$. The map \eqref{eq: finding a section of L_w} induces
    \[
        H^0(X,\,\mathcal F_1(1)) \otimes H^0(X,\,\mathcal F_2(1)) \otimes H^0(X,\,\mathcal F_3(1)) \otimes H^0(X,\,\mathcal F_5(1)) \otimes H^0(X,\,\mathcal F_{17}(1)) \to ^0(X,\,\mathcal F(3))
    \]
    under which we identify $s_{1,2,3,5,17}$ with $s_1s_2s_3s_5s_{17}$.
    
    On the other hand, we have further decompositions of $w$:
    \begin{align*}
        w &= w_1 + w_6 + w_{20} + w_{25} + w_{26} \\
        &= w_1 + w_7 + w_{19} + w_{23} + w_{24} \\
        &= w_2 + w_3 + w_4 + w_8 + w_{18} \\
        &= w_2 + w_7 + w_{13} + w_{14} + w_{17} \\
        &= w_3 + w_6 + w_9 + w_{10} + w_{17},
    \end{align*}
    from which we obtain five more sections 
    \[
        s_{1,6,20,25,26},\, s_{1,7,19,23,24},\,  \dots,\, s_{3,6,9,10,17} \in H^0(X,\,\mathcal F(3)).
    \]
    We claim that these six sections remain independent after passing through the restriction $H^0(X,\, \mathcal F(3)) \to H^0(H,\, \mathcal F(3) \big\vert_{\bar H})$, where $\bar H$ is the plane cutting out the half-even set $\mathcal N_1$. The proof of the claim is based on the following observation. For each $j=1,2,\dots,26$, let us denote the (Weil) divisor of $s_j$ by $C_j$. By Theorem~\ref{thm: Endrass min half-even set}, we have $2C_j = H_j . X$ for a plane $H_j \subset \PP^2$. For each node $P_i \in X$, either $P_i \not\in C_i$ or $P_i$ is a smooth point of $C_j$\,(see \cite[Lemma~2.3]{Endrass}). Thus, if
    \[
        w_j = \sum_{i=0}^{65} w_j^i e_i,\ \ w_j^i \in \FF_2,
    \]
    then $C_j$ contains $P_i$ if and only if $w_j^i=1$. The set $\bar H \cap \Sing X$ consists of the nodes
    \[
        P_3,\, P_8,\, P_{10},\, P_{19},\, P_{25},\, P_{34},\, P_{40},\, P_{42},\, P_{44},\, P_{47},\, P_{48},\, P_{49},\, P_{55},\, P_{61},\, P_{63}.
    \]
    We will show the linear independence in $H^0(X,\,\mathcal F(3)\big\vert_H)$ by looking at the vanishing orders of our sections at these points.
\begin{enumerate}
    \item The divisor of the section $s_{1,2,3,5,17}$ is
    \[
        C_1 + C_2 + C_3 + C_5 + C_{17}.
    \]
    From the corresponding codewords, we read
    \[
        \begin{array}{c|c|c|c|c}
             w_1^{40} & w_2^{40} & w_3^{40} & w_5^{40} & w_{17}^{40}  \\ \hline
             1 & 0 & 1 & 0 & 1 
        \end{array}\raisebox{-0.5\baselineskip}{,}
    \]
    showing that $C_1,\,C_3,\,C_{17}$ contains $P_{40}$, while $C_2,\,C_5$ does not.
    On the other hand, if we do the same analysis for $s_{1,6,20,25,26}$, then we find
    \[
        \begin{array}{c|c|c|c|c}
             w_1^{40} & w_6^{40} & w_{20}^{40} & w_{25}^{40} & w_{26}^{40}  \\ \hline
             1 & 0 & 0 & 0 & 0
        \end{array}\raisebox{-0.5\baselineskip}{.}
    \]
    Suppose $ \alpha s_{1,2,3,5,17}\big\vert_{\bar H} + \beta s_{1,6,20,25,26}\big\vert_{\bar H} = 0$ in $H^0(X,\,\mathcal F(3)\big\vert_{\bar H})$ for $\alpha,\beta \in \CC$. Since
    \[
        \alpha s_{1,2,3,5,17}\big\vert_{\bar H} + \beta s_{1,6,20,25,26}\big\vert_{\bar H} = s_1( \alpha s_2s_3s_5s_{17} + \beta s_6s_{20}s_{25}s_{26})\big\vert_{\bar H},
    \]
    $s_2s_3s_5s_{17}(P_{40}) = 0$ and $s_6 s_{20} s_{25}s_{26}(P_{40}) \neq 0$ imply $\alpha=\beta=0$. It follows that $s_{1,2,3,5,17}\big\vert_{\bar H}$ and $s_{1,6,20,25,26}\big\vert_{\bar H}$ are linearly independent.
    
    \item Let $V_1 \subset H^0(X,\,\mathcal F(3)\big\vert_H)$ be the subspace spanned by $\{ s_{1,2,3,5,17}\big\vert_{\bar H} , \, s_{1,6,20,25,26}\big\vert_{\bar H} \}$, and claim that $V_1$ does not contain $s_{1,7,19,23,24}\big\vert_{\bar H}$. At the point $P_{42}$, we have
    \[
        (w^{42}_1,\,w^{42}_2,\,w^{42}_3,\,w^{42}_5,\,w^{42}_{17}) = (1,1,0,0,1),\quad 
        (w^{42}_1,\,w^{42}_6,\,w^{42}_{20},\,w^{42}_{25},\,w^{42}_{26}) = (1,0,0,1,1),\quad 
    \]
    while
    \[
        (w^{42}_1,\,w^{42}_7,\,w^{42}_{19},\,w^{42}_{23},\,w^{42}_{24}) = (1,0,0,0,0).
    \]
    Every section in $V_1$ is of the form $s_1 v\big\vert_{\bar H}$ with $v(P_{42})=0$, while $s_7s_{19}s_{23}s_{24}(P_{42}) \neq 0$. This proves that $s_{1,7,19,23,24} \big\vert_{\bar H} \not\in V_1$.

    \item The remainder of the proof is essentially the repetition of the previous arguments. Let $V_2$ be the span of $V_1$ and $s_{1,7,19,23,24}\big\vert_{\bar H}$. From the corresponding codewords, we see that the sections in $V_2$ vanish at $P_{10}$, while $s_{2,3,4,8,18}(P_{10}) \neq 0$. Thus, $s_{2,3,4,8,18}\big\vert_{\bar H} \not\in V_2$. Let $V_3$ be the span of $V_2$ and $s_{2,3,4,8,18}\big\vert_{\bar H}$. The sections in $V_3$ vanish at $P_{48}$, while $s_{2,3,4,8,18}(P_{48}) \neq 0$, hence $s_{2,3,4,8,18}\big\vert_{\bar H} \not\in V_3$. Finally, if $V_4$ is the span of $V_3$ and $s_{2,3,4,8,18}\big\vert_{\bar H}$, then the sections in $V_4$ vanish at $P_{47}$, while $s_{3,6,9,10,17}(P_{47}) \neq 0$. Consequently, we see that $V_5:=(\text{the span of $V_4$ and $s_{3,6,9,10,17}\big\vert_{\bar H}$)}$ is $6$-dimensional.
\end{enumerate}
By Corollary~\ref{cor: key long exact seq}, we have $V_5 = H^0(X,\, \mathcal F(3) \big\vert_{\bar H})$. Clearly, the image of $H^0(X,\, \mathcal F(3)) \to H^0(X,\, \mathcal F(3) \big\vert_{\bar H})$ contains $V_5$, hence the result follows.
\end{proof}

\begin{proof}[Proof of Theorem~\ref{thm: main}]
    Let $\mathcal N$ and $\bar H$ be as in Proposition~\ref{prop: F(3) surjective}. By Corollary~\ref{cor: key long exact seq}, we have $h^0(X,\,\mathcal F(2))=0$. It follows from Proposition~\ref{prop: symmetry criterion} that $\mathcal N$ is symmetric of type $(1,1,1,1,1,1)$.
\end{proof}

\medskip
\begin{example}
Let $X$ be the Barth sextic surface defined by \eqref{eq: Barth sextic}. We have proved that the half-even set $w = w_1+w_2+w_3+w_5+w_{17}$ is symmetric of type $(1,1,1,1,1,1)$. By applying the method in \cite[Section~4.4]{CatChoKie}, we can compute the matrix $A$ for this particular instance. 

The matrix $A$ in Theorem~\ref{thm: main} is
\begingroup\renewcommand{\arraystretch}{1.4}
\[
    \scalebox{0.7}{$\left(
    \begin{array}{cccccc}
        (x_0+x_1+x_2-3x_3) &  -\frac{1}{3} (x_1 + \tau x_2 - \tau^2 x_3) & \frac{1}{3}( \bar\tau x_1 + x_2 - \bar \tau^2 x_3) & -\frac{3}{2}( \tau x_1 - \bar \tau x_2 - \sqrt 5 x_3) & \frac{1}{2}( x_1 - \tau x_2 - \tau^2 x_3) & \frac{1}{2}( \bar\tau x_1 - x_2 + \bar \tau^2 x_3) \\
        & 0 & 0 & -\tau(x_0-x_1-x_2+x_3) & \frac 13 \tau(x_0-x_1+x_2+x_3) & 0 \\
        & & 0 &  -\bar \tau (x_0-x_1-x_2+x_3) & 0 & \frac 13 \bar\tau(x_0+x_1-x_2+x_3) \\
        & & & 0 &  -\frac 32 (x_0-x_1-x_2-x_3) & -\frac 32 (x_0-x_1-x_2-x_3) \\
        & & & & 0 & \frac12 (x_0-x_1-x_2-x_3) \\
        & & & & & 0
    \end{array}
    \right)$}\raisebox{-2.8\baselineskip}{,}
\]\endgroup
where $\bar\tau = \frac 12(1-\sqrt 5)$ and the lower half triangle is symmetric to the upper half triangle. 
In particular, we have
\[
    \det A = -\frac{4}{9}\bar\tau^3 F.
\]
We refer the reader to \cite[Section~4.4]{CatChoKie} for the details how to find $A$ from $w$.
\end{example}

\end{document}